\documentclass[11pt,a4paper]{article}

\usepackage{epsf,epsfig,amsfonts,amsgen,amsmath,amstext,amsbsy,amsopn,amsthm,cases,listings,color
}
\usepackage{ebezier,eepic}
\usepackage{color}
\usepackage{multirow}
\usepackage{epstopdf}
\usepackage{graphicx}
\usepackage{pgf,tikz}
\usepackage{mathrsfs}
\usepackage[marginal]{footmisc}
\usepackage{enumitem}
\usepackage[titletoc]{appendix}
\usepackage{booktabs}
\usepackage{url}
\usepackage{mathtools}

\usepackage{pgfplots}
\usepackage{authblk}
\usepackage{amssymb}

\usepackage{wasysym}

\usepackage{empheq}

\usepackage{dsfont}

\usepackage{caption}
\usepackage{subcaption}
\pgfplotsset{compat=1.18}
\usepackage{mathrsfs}
\usepackage{wasysym} 
\usetikzlibrary{arrows}
\usepackage{aligned-overset}
\usepackage{bm}
\usepackage[backref=page]{hyperref}
\usepackage{makecell}

\newcommand{\Mod}[1]{\ (\mathrm{mod}\ #1)}

\allowdisplaybreaks[1]

\definecolor{uuuuuu}{rgb}{0.27,0.27,0.27}
\definecolor{sqsqsq}{rgb}{0.1255,0.1255,0.1255}

\newtheorem{definition}{Definition} [section]
\newtheorem{theorem}[definition]{Theorem}
\newtheorem{lemma}[definition]{Lemma}
\newtheorem{proposition}[definition]{Proposition}

\newtheorem{corollary}[definition]{Corollary}

\newtheorem{claim}[definition]{Claim}
\newtheorem{problem}[definition]{Problem}

\newtheorem{fact}[definition]{Fact}
\newenvironment{pf}{\noindent{\bf Proof}}

\begin{document}
%%%%%%%%%%%%%%%%%%%%%%%%%%%%%%%%%%%%%%%%%%%%%%%%%%%%%%%
\title{\bf\Large Tight bounds for rainbow partial $F$-tiling in edge-colored complete hypergraphs}
\date{\today}
%%%%%%%%%%%%%%%%%%%%%%%%%%%%%%%%%%%%%%%%%%%%%%%%%%%%
\author[1]{Jinghua Deng\thanks{Email: \texttt{Jinghua\_deng@163.com}}}
\author[1]{Jianfeng Hou\thanks{Research was supported by National Key R$\&$D Program of China (Grant No. 2023YFA1010202), National Natural Science Foundation of China (Grant No. 12071077) and the Central Guidance on Local Science and Technology Development Fund of Fujian Province (Grant No. 2023L3003). Email: \texttt{jfhou@fzu.edu.cn}}}
\author[2]{Xizhi Liu\thanks{Research was supported by ERC Advanced Grant 101020255. \texttt{Email: xizhi.liu.ac@gmail.com}}}
% \author[ ]{Bailey Zou\thanks{Email: \texttt{baileyzoe.ac@gmail.com}}}
\author[1]{Caihong Yang\thanks{Email: \texttt{chyang.fzu@gmail.com}}}
%%%%%%%%%%%%%%%%%%%%%%%%%%%%%%%%%%%%%%%%%%%%%%%%%%%%%
\affil[1]{Center for Discrete Mathematics,
            Fuzhou University, Fujian, 350003, China}
\affil[2]{Mathematics Institute and DIMAP,
            University of Warwick,
            Coventry, CV4 7AL, UK}
%%%%%%%%%%%%%%%%%%%%%%%%%%%%%%%%%%%%%%%%%%%%%%%%%%%
\maketitle
%\footnote{footnote}
%%%%%%%%%%%%%%%%%%%%%%%%%%%%%%%%%%%%%%%%%%%%%%%%%
%%%%%%%%%%%%%%%%%%%%%%%%%%%
\begin{abstract}
    For an $r$-graph $F$ and integers $n,t$ satisfying $t \le n/v(F)$, let $\mathrm{ar}(n,tF)$ denote the minimum integer $N$ such that every edge-coloring of $K_{n}^{r}$ using $N$ colors contains a rainbow copy of $tF$, 
    where $tF$ is the $r$-graphs consisting of $t$ vertex-disjoint copies of $F$. 
    The case $t=1$ is the classical anti-Ramsey problem proposed by Erd\H{o}s--Simonovits--S\'{o}s~\cite{ESS75}. 
    When $F$ is a single edge, this becomes the rainbow matching problem introduced by Schiermeyer~\cite{Sch04} and \"{O}zkahya--Young~\cite{OY13}. 

    We conduct a systematic study of $\mathrm{ar}(n,tF)$ for the case where $t$ is much smaller than $\mathrm{ex}(n,F)/n^{r-1}$. 
    Our first main result provides a reduction of $\mathrm{ar}(n,tF)$ to $\mathrm{ar}(n,2F)$ when $F$ is bounded and smooth, two properties satisfied by most previously studied hypergraphs. 
    Complementing the first result, the second main result, which utilizes gaps between Tur\'{a}n numbers, determines $\mathrm{ar}(n,tF)$ for relatively smaller $t$. 
    Together, these two results determine $\mathrm{ar}(n,tF)$ for a large class of hypergraphs. 
    Additionally, the latter result has the advantage of being applicable to hypergraphs with unknown Tur\'{a}n densities, such as the famous tetrahedron $K_{4}^{3}$. 

\medskip

\textbf{Keywords:} anti-Ramsey problem, Tur\'{a}n problem, hypergraphs, boundedness, smoothness,  stability.
% \medskip
% \textbf{MSC2020:} 	05C65, 05C35, 05D99. 
%Find suitable code from https://mathscinet.ams.org/msc/msc2010.html
\end{abstract}
%%%%%%%%%%%%%%%%%%%%%%%%%%%%%%%%%%%%%%%%%%%%%%%%%%%%%%
\section{Intorduction}\label{SEC:Intorduction}
% \subsection{Motivation}
Given an integer $r\ge 2$, an \textbf{$r$-uniform hypergraph} (henceforth \textbf{$r$-graph}) $\mathcal{H}$ is a collection of $r$-subsets of some finite set $V$.
We identify a hypergraph $\mathcal{H}$ with its edge set and use $V(\mathcal{H})$ to denote its vertex set. 
The size of $V(\mathcal{H})$ is denoted by $v(\mathcal{H})$. 
The $n$-vertex complete $r$-graph is denoted by $K_n^{r}$.
The superscript $r$ will be omitted in the case where $r=2$. 

Given an $r$-graph $F$ and an integer $t \ge 1$, let $t F$ denote the $r$-graph consisting of $t$ vertex-disjoint copies of $F$. 
We say that an edge-colored $tF$ is \textbf{rainbow} if no pair of edges in $tF$ share the same color. 
Motivated by the classical \textbf{anti-Ramsey problem} proposed by Erd\H{o}s--Simonovits--S\'{o}s~\cite{ESS75}, as well as theorems of Schiermeyer~\cite{Sch04} and \"{O}zkahya--Young~\cite{OY13} on rainbow matchings, we aim to initiate a systematic study of the following problem in this work. 
\begin{problem}\label{PROB:antiRmasey-tF}
    Given an $r$-graph $F$ and positive integers $n, t$ satisfying $t \le n/v(F)$, determine the minimum value $N$, denoted by $\mathrm{ar}(n,tF)$, such that every surjective edge-coloring $\chi \colon K_{n}^{r} \to [N]$ contains a rainbow copy of $tF$. 
\end{problem}
Before surveying related results, let us introduce some definitions related to Tur\'{a}n problems. 
Given a family $\mathcal{F}$ of $r$-graphs, we say an $r$-graph $\mathcal{H}$ is \textbf{$\mathcal{F}$-free} if it does not contain any member of $\mathcal{F}$ as a subgraph. 
The \textbf{Tur\'{a}n number} $\mathrm{ex}(n, \mathcal{F})$ of $\mathcal{F}$ is the maximum number of edges in an $\mathcal{F}$-free $r$-graph on $n$ vertices. 
The \textbf{Tur\'{a}n density} of $\mathcal{F}$ is defined as $\pi(\mathcal{F})\coloneq \lim_{n\to\infty}\mathrm{ex}(n,\mathcal{F})/{n\choose r}$. 
A family $\mathcal{F}$ of $r$-graphs is called \textbf{nondegenerate} if $\pi(\mathcal{F})>0$, and \textbf{degenerate} otherwise.
Determining $\mathrm{ex}(n, \mathcal{F})$, even asymptotically, is a central topic in Extremal Combinatorics. 
For most hypergraphs and degenerate graphs, we do not know the asymptotic behavior of $\mathrm{ex}(n, \mathcal{F})$. 
Some notoriously hard problems in this area include the Tetradron Conjecture, proposed by Tur\'{a}n~\cite{TU41} over 80 years ago, which asserts that $\pi(K_{4}^{3}) = 5/9$, and the Even Cycle Problem, proposed by Erd\H{o}s~\cite{E64,BS74}, which asks for the exponent of $\mathrm{ex}(n,C_{2k})$.  
Here, we refer the reader to surveys~\cite{Kee11,FS13} for further results on Tur\'{a}n problems. 

The case $t=1$ in Problem~\ref{PROB:antiRmasey-tF}  corresponds to the classical anti-Ramsey problem introduced by Erd\H{o}s--Simonovits--S\'{o}s in~\cite{ESS75}. 
They established the following general relation between the anti-Ramsey number and the Tur\'{a}n number of an $r$-graph $F$$\colon$
\begin{align*}%\label{equ:antiRamsey-general}
    \mathrm{ex}(n, F_-)+2 
    \le \mathrm{ar}(n, F) 
    \le \mathrm{ex}(n, F_-)+o(n^r).
\end{align*}
Here, $F_-$ denotes the family of $r$-graphs obtained from $F$ by removing exactly one edge. 
Determining the precise form of the error term in the upper bound is a central topic in anti-Ramsey theory. 
For complete graphs, Erd\H{o}s--Simonovits--S\'{o}s proved that for fixed $\ell\ge 2$ and sufficiently large $n$,
\begin{align*}%\label{equ:anti-ramsey-K}
    \mathrm{ar}(n, K_{\ell+1})
    = \mathrm{ex}(n, K_\ell) + 2.
\end{align*}
The restriction on $n$ was later removed by Montellano-Ballesteros and Neumann-Lara in~\cite{MN02}.
Numerous results on graph anti-Ramsey problems have been obtained by various researchers over the past five decades. 
For related results, we refer the reader to the survey by Fujita--Magnant--Ozeki~\cite{FMO10}. 
On the other hand, the understanding of hypergraph anti-Ramsey problems is mostly limited to very sparse and special hypergraphs such as linear paths, linear cycles, and linear trees (see e.g.~\cite{GLS20,TLY22,LS23a}). 
Results on nondegenerate hypergraphs, which extends the Erd\H{o}s--Simonovits--S\'{o}s Theorem on complete graphs, were obtained only very recently~\cite{LS23b,LTY24}. 
% Given our limited understanding of hypergraph Tur\'{a}n problems, understanding the hypergraph anti-Ramsey problems 

For $t \ge 2$, there are only a few results on Problem~\ref{PROB:antiRmasey-tF}. 
A theorem of Jiang--Pikhurko~\cite{JP09} on doubly edge-critical graphs determines $\mathrm{ar}(n,2F)$ when $F$ is an edge-critical graph and $n$ is large, where edge-critical means there exists an edge whose removal will decrease the chromatic number. 
The theorem of Jiang--Pikhurko was extended to hypergraphs very recently~\cite{LS23b,LTY24}, and hence, $\mathrm{ar}(n,2F)$ is also known for a certain class of hypergraphs. 
Schiermeyer~\cite{Sch04} and \"{O}zkahya--Young~\cite{OY13} initiated the study of Problem~\ref{PROB:antiRmasey-tF} for $F = K_{r}^{r}$ with $r=2$ and $r\ge 3$, respectively. 
The case $F = K_{3}$ was studied very recently by Wu--Zhang--Li--Xiao in~\cite{WZLX23}.  
It appears that $K_2$ is the only example, even for large $n$, for which Problem~\ref{PROB:antiRmasey-tF} has been completely solved~\cite{Sch04,FKSS09,CLT09,HY12}. 

In this work, we conduct a systematic study of Problem~\ref{PROB:antiRmasey-tF} for general $r$-graphs $F$, with a focus on the case where $n$ is sufficiently large and $2 \le t \ll \mathrm{ex}(n,F)/n^{r-1}$.  
Our first main result is a reduction theorem (Theorem~\ref{THM:antiRamsey-reduction}), which, for a large class of hypergraphs $F$, reduces the problem of determining $\mathrm{ar}(n,tF)$ to determining $\mathrm{ar}(n,2F)$ when $t \ll \mathrm{ex}(n,F)/n^{r-1}$. 
The second main result (Theorem~\ref{THM:main-antiRmasey-Turan-Gap}) utilizes the gap between Tur\'{a}n numbers $\mathrm{ex}(n,F)$ and $\mathrm{ex}(n,\{F\} \cup F\oplus F)$ (defined in Section~\ref{SEC:Intro-Turan-Gap}) to determine $\mathrm{ar}(n,tF)$ for relatively small $t$. 
This theorem, combined with the first main result, determines $\mathrm{ar}(n,tF)$ for numerous hypergraphs (see Table~\ref{tab:turan-edge-critical.}) when $n$ is large and $2 \le t \ll \mathrm{ex}(n,F)/n^{r-1}$.
In addition, the second main theorem also determines $\mathrm{ar}(n,tF)$ for numerous hypergraphs with unknown Tur\'{a}n densities (see Table~\ref{tab:turan-edge-critical-2.}) when $n$ is large and $2 \le t \ll \sqrt{n}$. This include the famous $K_{4}^{3}$. 

\textbf{Notations.}
For an $r$-graph $\mathcal{H}$ and a vertex $v\in V(\mathcal{H})$, the \textbf{link} of $v$ in $\mathcal{H}$ is 
\begin{align*}
    L_{\mathcal{H}}(v)
    \coloneqq \left\{e \in \binom{V(\mathcal{H})}{r-1} \colon e\cup \{v\} \in \mathcal{H} \right\}. 
\end{align*}
The \textbf{degree} of $v$ in $\mathcal{H}$ is $d_{\mathcal{H}}(v) \coloneqq |L_{\mathcal{H}}(v)|$.
The \textbf{maximum degree} of $\mathcal{H}$ is denoted by $\Delta(\mathcal{H})$. 

For an $r$-graph $\mathcal{H}$ and a vertex set $S\subseteq V(\mathcal{H})$, we use $\mathcal{H}[S]$ to denote the induced subgraph of $\mathcal{H}$ on $S$, 
and use $\mathcal{H}-S$ to denote the induced subgraph of $\mathcal{H}$ on $V(\mathcal{H}) \setminus S$. 

Throughout this paper, asymptotic notations are taken with respect to $n$. 
%%%%%%%%%%%%%%%%%%%%%%%%%%%%%%%%%%%%%%%%%%%%%%
\subsection{The reduction theorem}
In this subsection, we present the reduction theorem. 
To state the main result we need some definitions from~\cite{HLLYZ23,HHLLYZ23a}. 
% on anti-ramsey number of $t+1$ disjoint copies of $F$ where $F$ have good boundedness and smoothness. We denote by $(t+1) F$ the $t+1$ disjoint copies of $F$. Before stating our results, let us introduce some necessary
%  definitions.

Let $F$ be an $r$-graph. For every $n\in\mathbb N$, let
\begin{align*}
    \delta(n,F) 
    \coloneq \mathrm{ex}(n,F)-\mathrm{ex}(n-1,F)
    \quad\text{and}\quad 
    d(n,F)
    \coloneq \frac{r\cdot \mathrm{ex}(n,F)}{n}.
\end{align*}
\begin{definition}
    Let $r \ge 2$ be an integer and $F$ be an $r$-graph. 
    \begin{itemize}
        \item For two real numbers $c_1, c_2 > 0$, we say $F$ is \textbf{$(c_1, c_2)$-bounded} if there exists $N_0$ such that every $r$-graph $\mathcal{H}$ on $n\ge N_0$ vertices with
          \begin{align*}
              \Delta(\mathcal{H}) 
              \ge d(n,F) + c_1\binom{n-1}{r-1} 
              \quad\text{and}\quad
              |\mathcal{H}| 
              \ge (1-c_2) \cdot \mathrm{ex}(n,F)  
          \end{align*}
          contains a copy of $F$.  
          We say $F$ is \textbf{bounded} if it is $(c_1, c_2)$-bounded for some constants $c_1, c_2 > 0$.
        \item We say $F$ is \textbf{smooth} if it is degenerate\footnote{Since our results do not require~\eqref{equ:def-smooth} for degenerate $r$-graphs, we consider every degenerate $r$-graph as smooth by default for convenience.} or there exists $N_0$ such that 
            \begin{align}\label{equ:def-smooth}
                |\delta(n,F)-d(n-1,F)|\le \frac{1-\pi(F)}{8 v(F)}\binom{n}{r-1} 
                \quad\text{for all $n \ge N_0$}. 
             \end{align}
            \end{itemize}
\end{definition}
Boundedness and smoothness, introduced in~\cite{HLLYZ23,HHLLYZ23a}, are crucial properties for determining $\mathrm{ex}(n,tF)$, the maximum number of edges in an $n$-vertex $r$-graph without $tF$. 
It was shown in~\cite{HLLYZ23} and ~\cite{HHLLYZ23c} that many well-studied nondegenerate and degenerate $r$-graphs (see Table~\ref{tab:turan-edge-critical.}) are bounded and smooth, respectively. 
% It was proved in~\cite{HHLLYZ23c} that many well-studied degenerate $r$-graphs, including even cycles and complete bipartite graphs $K_{s,t}$ with $t \ge s^2$, are bounded. 

The following theorem presents an application of boundedness and smoothness in anti-Ramsey theory.  
% Since our results do not require the smoothness property for degenerate $r$-graphs, it will be convenient to consider every degenerate $r$-graph as smooth by default.

\begin{theorem}\label{THM:antiRamsey-reduction}
    Let $r \ge 2$ be an integer. 
    Suppose that $F$ is an $r$-graph that is smooth and $(c_1,c_2)$-bounded for constants $c_1,c_2$ satisfying $0 < c_1 \le \frac{1-\pi(F)}{12m}$ and $c_2>0$. 
    Then there exists $\delta >0$ and $N_0$ such that for every $n \ge N_0$ and for every $t \in \left[0, \frac{\delta\cdot \mathrm{ex}(n,F)}{n^{r-1}}\right]$, 
    \begin{align*}
        \mathrm{ar}(n,(t+2)F)
        = \binom{n}{r}-\binom{n-t}{r} + \mathrm{ar}(n-t,2F). 
    \end{align*}
\end{theorem}
Proof of Theorem~\ref{THM:antiRamsey-reduction} is presented in Section~\ref{SEC:proof-reduction}. 
%after we establish a stability theorem concerning near-extremal hypergraphs without $tF$ in Section~\ref{SEC:proof-stability}. 

%%%%%%%%%%%%%%%%%%%%%%%%%%%%%%%%%%
\subsection{Anti-Ramsey results from gaps between Tur\'{a}n numbers}\label{SEC:Intro-Turan-Gap}
%
% In this subsection, we give another form of the anti-ramsey number of $t+1$ disjoint $F$. The main theorem in this subsection applies to more graph classes than the one in Theorem~\ref{THM:antiRamsey-reduction}, but the range of values of $t$ is relatively small.
Complementing the theorem from the previous subsection, we present a theorem that addresses $\mathrm{ar}(n,tF)$ when $t$ is small, particularly for $t=2$, in this subsection. 

%%%%%%%%%%%%%%%%%%%%%%
%%%%%%%%%%%%%%%%%%%%%%%%%%%%%%%%%
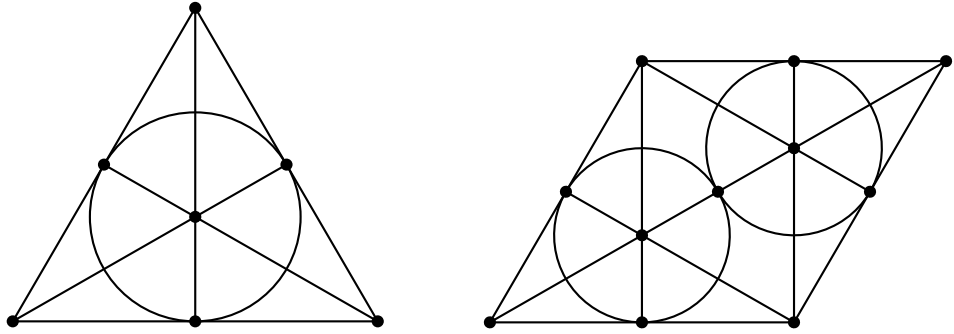
\begin{figure}[htbp]
\centering
\begin{subfigure}[c]{0.45\textwidth}
\centering
\begin{tikzpicture}[xscale=2.4,yscale=2.4]
\node (a) at (-1,0) {};
\fill (a) circle (0.0333);
\node (b) at (0,0) {};
\fill (b) circle (0.0333);
\node (c) at (1,0) {};
\fill (c) circle (0.0333);
\node (d) at (0.5,0.866025) {};
\fill (d) circle (0.0333);
\node (e) at (0,1.732051) {};
\fill (e) circle (0.0333);
\node (f) at (-0.5,0.866025) {};
\fill (f) circle (0.0333);
\node (g) at (0,0.577350) {};
\fill (g) circle (0.0333);
\node (h) at (0,-0.3) {};
\draw[line width=0.8pt] (0,0.577350) circle [radius = 0.577350];
\draw[line width=0.8pt] (-1,0) -- (1,0);
\draw[line width=0.8pt] (-1,0) -- (0,1.732051);
\draw[line width=0.8pt] (-1,0) -- (0.5,0.866025);
\draw[line width=0.8pt] (1,0) -- (0,1.732051);
\draw[line width=0.8pt] (1,0) -- (-0.5,0.866025);
\draw[line width=0.8pt] (0,0) -- (0,1.732051);
\end{tikzpicture}
% \caption{The Fano plane $\mathbb{F}$.}
% \label{fig:projective-plane-STS-7}
\end{subfigure}
\begin{subfigure}[c]{0.45\textwidth}
\centering
\begin{tikzpicture}[xscale=2,yscale=2,rotate=120]
\node (a) at (-1,0) {};
\fill (a) circle (0.04);
\node (b) at (0,0) {};
\fill (b) circle (0.04);
\node (c) at (1,0) {};
\fill (c) circle (0.04);
\node (d) at (0.5,0.866025) {};
\fill (d) circle (0.04);
\node (e) at (0,1.732051) {};
\fill (e) circle (0.04);
\node (f) at (-0.5,0.866025) {};
\fill (f) circle (0.04);
\node (g) at (0,0.577350) {};
\fill (g) circle (0.04);
\node (h) at (0,-0.3) {};
\draw[line width=0.8pt] (0,0.577350) circle [radius = 0.577350];
% \draw[line width=0.8pt] (-1,0) -- (1,0);
\draw[line width=0.8pt] (-1,0) -- (0,1.732051);
\draw[line width=0.8pt] (-1,0) -- (0.5,0.866025);
\draw[line width=0.8pt] (1,0) -- (0,1.732051);
\draw[line width=0.8pt] (1,0) -- (-0.5,0.866025);
\draw[line width=0.8pt] (0,0) -- (0,1.732051);
\draw[line width=0.8pt] (0,0) -- (0,-1.732051);
\draw[line width=0.8pt] (-1,0) -- (0,-1.732051);
\draw[line width=0.8pt] (1,0) -- (0,-1.732051);
\draw[line width=0.8pt] (1,0) -- (-0.5,-0.866025);
\draw[line width=0.8pt] (-1,0) -- (0.5,-0.866025);
\draw[line width=0.8pt] (0,-0.577350) circle [radius = 0.577350];
\node (d) at (0.5,-0.866025) {};
\fill (d) circle (0.04);
\node (e) at (0,-1.732051) {};
\fill (e) circle (0.04);
\node (f) at (-0.5,-0.866025) {};
\fill (f) circle (0.04);
\node (g) at (0,-0.577350) {};
\fill (g) circle (0.04);
% \node (h) at (0,-0.3) {};
\end{tikzpicture}
% \caption{$\mathbb{F}\oplus\mathbb{F}$.}
% \label{fig:affine-plane-STS-9}
\end{subfigure}
\caption{Fano plane $\mathbb{F}$ (left) and $\mathbb{F}\oplus\mathbb{F}$ (right). Here $\mathbb{F} \coloneqq \{123,345,561,174,275,376,246\}$.}
\label{Fig:Fano-plane}
\end{figure}
%%%%%%%%%%%%%%%%%%%%%%%%%%%%%%%%%
%%%%%%%%%%%%%%%%%%%%%%

Given an $r$-graph $F$ and an edge $e \in F$, let $F_{e-}$ denote the $r$-graph obtained from $F$ by removing the edge $e$. 
To be specific, $F_{e-} \coloneqq F\setminus \{e\}$. 
Recall that $F_- = \{F_{e-} \colon e\in F\}$. 
Let $F, F'$ be two $r$-graphs, $(e,e') \in F\times F'$ be a pair of edges, and $\phi \colon e \to e'$ be a bijection. 
Following the definition in~\cite{Wag37,KT90}, the \textbf{edge-sum} $F \oplus_{\phi} F'$ is the $r$-graph obtained from $F_{e-}$ and $F'_{e'-}$ by identifying $v$ with $\phi(v)$ for all $v\in e$. 
For convenience, let 
\begin{align*}
    F \oplus F'
    \coloneqq 
    \left\{F \oplus_{\phi} F' \colon \text{$\phi \colon e \to e'$ is bijective for some pair $(e,e') \in F\times F'$}\right\} 
\end{align*}
denote the collection of all edge-sums of $F$ and $F'$ (see Figure~\ref{Fig:Fano-plane}). 
A quick observation is that $C_{k} \oplus C_{\ell} = \{C_{k+\ell-2}\}$ for all $k, \ell \ge 3$. 

For convenience, we slightly abuse notation by letting $\{F\}\cup F\oplus F \coloneqq \{F\}\cup \left(F\oplus F\right)$. 
Note that $\mathrm{ex}(n,F) - \mathrm{ex}\left(n, \{F\}\cup F\oplus F \right) \ge 0$ for every $F$ and for every $n \in \mathbb{N}$. 
\begin{definition}\label{DDF:turan-edge-critical}
    An $r$-graph $F$ is \textbf{edge-sensitive} if there exists $N_0$ such that 
    \begin{align*}%\label{equ:edge-sum-turan-number}
        \mathrm{ex}(n,F) - \mathrm{ex}\left(n, \{F\}\cup F\oplus F \right)
        \ge 2 v(F) |F| \binom{n-1}{r-1}
        \quad\text{for all}\quad n \ge N_0. 
    \end{align*}
\end{definition}
%
% Recall that a graph $F$ is \textbf{edge-critical} if there exists $e\in F$ such that the chromatic number of $F\setminus\{e\}$ is strictly smaller than that of $F$. 
The definition of edge-sensitive is motivated by that of edge-critical.  
It follows easily from the classical Erd\H{o}s--Stone--Simonovits Theorem~\cite{ES66} that every edge-critical graph is edge-sensitive. 
However, as shown by the even cycle $C_{2k}$, the converse does not hold (see Proposition~\ref{PROP:turan-edge-critical}~\ref{PROP:turan-edge-critical-2}). 

The following theorem provides an application of the gap between $\mathrm{ex}(n,F)$ and $\mathrm{ex}(n, \{F\} \cup F\oplus F)$ in anti-Ramsey theory. 
\begin{theorem}\label{THM:main-antiRmasey-Turan-Gap}
    Let $r \ge 2$ be an integer and $F$ be an $r$-graph. 
    For every $n \in \mathbb{N}$, 
    \begin{align*}
        \mathrm{ar}(n,(t+1)F) = \mathrm{ex}(n,tF) + 2
        \quad\text{for}\quad 
        1 \le t \le \sqrt{\frac{\mathrm{ex}(n,F) - \mathrm{ex}\left(n, \{F\}\cup F\oplus F\right)}{2 v(F) |F| \binom{n-1}{r-1}}}
    \end{align*}
    % for all $n \ge N_0$ and $1 \le t \le \sqrt{\frac{\mathrm{ex}(n,F) - \mathrm{ex}\left(n, \{F\}\cup F\oplus F\right)}{2 v(F) |F| \binom{n-1}{r-1}}}$. 
    In particular, if $F$ is edge-sensitive, then for sufficiently large $n$, 
    \begin{align*}
        \mathrm{ar}(n,2F) = \mathrm{ex}(n,F) + 2.
    \end{align*}
\end{theorem}
Proof of Theorem~\ref{THM:main-antiRmasey-Turan-Gap} is presented in Section~\ref{SEC:proof-Turan-gap}. 

%%%%%%%%%%%%%%%%%%%%%%%%%%%%
\subsection{Applications}\label{SEC:Intro-application}
In this subsection, we present some applications of Theorems~\ref{THM:antiRamsey-reduction} and \ref{THM:main-antiRmasey-Turan-Gap}. 
Let us first present a simple result concerning Tur\'{a}n numbers. 
\begin{proposition}\label{PROP:turan-edge-critical}
    The following statements hold. 
    \begin{enumerate}[label=(\roman*)]
        \item \label{PROP:turan-edge-critical-1} For every $r$-graph $F$, if $\pi(F_{-}) < \pi(F)$, then $F$ is edge-sensitive. 
        \item\label{PROP:turan-edge-critical-2} For every integer $k \ge 2$, the even cycle $C_{2k}$ is edge-sensitive. 
          \item\label{PROP:turan-edge-critical-3} For the Fano Plane $\mathbb{F}$, we have $\pi(\mathbb{F}_{-}) < \pi(\mathbb{F})$.
    \end{enumerate}
\end{proposition}
In the following table, we summarize hypergraphs that exhibit boundedness/smoothness and are edge-sensitive. 
Proofs for the boundedness/smoothness can be found in either~\cite{HLLYZ23} (for nondegenerate hypergraphs) or~\cite{HHLLYZ23c} (for degenerate hypergraphs). 
Definitions for these hypergraphs are included in the Appendix. 
%%%%%%%%%%%%%%%%%%%%%
\begin{table}[h]
    \centering
    \begin{tabular}{c|c}
        \hline
        Hypergraphs & References \\
        \hline
        Even cycle $C_{2k}$ for $k \ge 2$ & ~\cite{BS74,LUW95} \\ 
        Edge-critical graphs & ~\cite{AES74,ES73}  \\
        % Non-edge-critical graphs~\cite{LMR23unif} \\
        Expansion of edge-critical graphs & ~\cite{MU06,PI13}  \\
        % Expansion of non-edge-critical graphs  \\
        Expansion of extended Erd\H{o}s--S\'{o}s tree & ~\cite{Sido89,NY18,BIJ17}  \\
        Expansion of $M_{2}^r$ for $r \ge 3$ & ~\cite{HK13,BNY19}  \\
        Expansion of $M_{k}^3$, $L_{k}^3$, or $L_{k}^4$ for $k \ge 2$ & ~\cite{HK13,JPW18} \\
        Expansion of $M_{k}^{4}$ for $k \ge 2$ & ~\cite{YP23}  \\
        Expansion of $K_{4}^{3} \sqcup K_{3}^{3}$ & ~\cite{YP22}  \\
        Generalized triangle $\mathbb{T}_r$ for $3 \le r \le 6$ & ~\cite{BO74,FF83,Sido87,FF89,PI08}\\
        Expanded triangle $\mathcal{C}_{3}^{2r}$ for $r \ge 2$ & ~\cite{Frankl90,KS05a} \\
        $\mathbb{F}_{3,2}$ ($3$-book with $3$ pages) & ~\cite{FPS053Book3page}  \\
        $\mathrm{F}_7$ ($4$-book with $3$ pages) & ~\cite{FPS06Book} \\
        $\mathbb{F}_{4,3}$ ($4$-book with $4$ pages) & ~\cite{FMP084Book4page} \\
        % Expansion of $M_{3}^5$  and $M_{4}^5$~\cite{HHW24} & Yes \\
         Fano Plane $\mathbb{F}$ & ~\cite{DF00,KS05,FS05}  \\
        % \xl{need check} $\Sigma_r$ for $r\ge 2$ & ~\cite{FF89, Mantel07, Bol74, Sido87, NY17}\\
        % \xl{need check} $\Delta_r$ for $r=2,3,4$ & ~\cite{Mantel07,Bol74,TU41,Kat74,Sido87}\\
        \hline
    \end{tabular}
    \caption{Edge-sensitive hypergraphs with smoothness and boundedness.}
    \label{tab:turan-edge-critical.}
\end{table}
%%%%%%%%%%%%%%%%%%%%%

In the following table, we summerize edge-sensitive hypergraphs for which the boundedness and smoothness are unknown. 
%%%%%%%%%%%%%%%%%%%%%
\begin{table}[h]
    \centering
    \begin{tabular}{c|c}
        \hline
        Hypergraphs & References \\
        \hline
        $K_{4}^{3-}$ & ~\cite{FF84,BT11} \\
        $K_{4}^{3}$ & ~\cite{TU41,RA10} \\
        $K_{5}^{4}$ &  ~\cite{Gir90,Mar09} \\
        Tight cycle $C_{3\ell+1}^{3}$ and $C_{3\ell+2}^{3}$ for $\ell \ge 1$  & ~\cite{RA10,MPS11,FV12,KLP22} \\
        $C_{3\ell+1}^{3-}$ and $C_{3\ell+2}^{3-}$ for $\ell \ge 1$ & ~\cite{MPS11,BL24} \\
        Generalized triangle $\mathbb{T}_{r}$ for $r \ge 7$  & ~\cite{BO74,FF89} \\
        % $\Sigma_{r}$ for $r \ge ?$  & \\
        % $\Delta_{r}$ for $r \ge ?$  & \\
        \hline
    \end{tabular}
    \caption{Edge-sensitive hypergraphs whose boundedness/smoothness are unknown.}
    \label{tab:turan-edge-critical-2.}
\end{table}
%%%%%%%%%%%%%%%%%%%

Having completed these preparations, we can now present the corollary of Theorems~\ref{THM:antiRamsey-reduction} and~\ref{THM:main-antiRmasey-Turan-Gap}. 
\begin{corollary}\label{CORO:anti-Ramsey-tF}
    Let $r \ge 2$ be an integer and $F$ be an $r$-graph. There exist $\delta>0$ and $N_0$ such that the following statements hold for all $n \ge N_0$. 
    \begin{enumerate}[label=(\roman*)]
        \item\label{CORO:anti-Ramsey-tF-1} If $F$ is contained in Table~\ref{tab:turan-edge-critical.}, then 
        \begin{align*}
            \mathrm{ar}(n,(t+2)F)
            = \binom{n}{r} - \binom{n-t}{r} + \mathrm{ex}(n,F) + 2
            \quad\text{for every}\quad 
            0 \le t \le \frac{\delta \cdot \mathrm{ex}(n,F)}{n^{r-1}}.
        \end{align*}
        \item\label{CORO:anti-Ramsey-tF-2} If $F$ is contained in Table~\ref{tab:turan-edge-critical-2.}, then 
        \begin{align*}
            \mathrm{ar}(n,(t+1)F)
            = \mathrm{ex}(n,tF) + 2
            \quad\text{for every}\quad 
            1 \le t \le \delta \sqrt{n}.
        \end{align*}
        \item\label{CORO:anti-Ramsey-tF-3} If $F$ is nondegenerate and contains an edge $e \in F$ such that $F\setminus\{e\}$ is $r$-partite, then 
        \begin{align*}
            \mathrm{ar}(n,(t+1)F)
            = \mathrm{ex}(n,tF) + 2
            \quad\text{for every}\quad 1 \le t \le \delta \sqrt{n}.
        \end{align*}
    \end{enumerate}
\end{corollary}
Corollary~\ref{CORO:anti-Ramsey-tF}~\ref{CORO:anti-Ramsey-tF-3} follows easily from a theorem of Erd\H{o}s~\cite{Erdos64}, which implies that an $r$-graph is degenerate iff it is $r$-partite. 

Proof of Proposition~\ref{PROP:turan-edge-critical} is presented in Section~\ref{SEC:proof-Fano-minus}. 

%%%%%%%%%%%%%%%%%%%%%%%%%%%%%%%%%%%%%%%%%%%%%%%%%%
\section{Proof of Theorem \ref{THM:main-antiRmasey-Turan-Gap}}\label{SEC:proof-Turan-gap}
In this section, we prove Theorem~\ref{THM:main-antiRmasey-Turan-Gap}. The lower bound in Theorem~\ref{THM:main-antiRmasey-Turan-Gap} comes from the following simple fact, which holds for all $r$-graphs. 
\begin{fact}\label{FACT:antiRamsey-tF-general-bound}
    For every $r$-graph and positive integers $n,t$ with $1 \le t \le n/v(F)$, 
    \begin{align*}
        \mathrm{ex}(n,tF)+2
        \le \mathrm{ar}(n,(t+1)F) 
        \le \mathrm{ex}(n,(t+1)F)+1. 
    \end{align*}
\end{fact}
\begin{proof}[Proof of Fact~\ref{FACT:antiRamsey-tF-general-bound}]
    The upper bound in Fact~\ref{FACT:antiRamsey-tF-general-bound} is trivial, so it suffices to prove the lower bound. 
    Let $N \coloneqq \mathrm{ex}(n,tF)$. 
    Fix a $tF$-free subgraph $\mathcal{H} \subset K_{n}^{r}$ with exactly $\mathrm{ex}(n,tF)$ edges. Let $\chi \colon K_{n}^{r} \to \mathbb{N}$ be an edge-coloring satisfying 
    \begin{itemize}
        \item the induced coloring of $\chi$ on $\mathcal{H}$ is rainbow and use colors in $[N]$, 
        \item the induced coloring of $\chi$ on $K_{n}^{r} \setminus \mathcal{H}$ uses only the color $N+1$. 
    \end{itemize}
    Suppose to the contrary that there exists a rainbow copy of $(t+1)F$ in $K_{n}^{r}$ under $\chi$. Then at most one $F$ in this $(t+1)F$ can use the color $N+1$, meaning that there is a rainbow copy of $tF$ using colors from $[N]$. However, this is impossible since $\mathcal{H}$ is $tF$-free. 
\end{proof}
Next, we prove the upper bound in Theorem~\ref{THM:main-antiRmasey-Turan-Gap}. 
\begin{proof}[Proof of Theorem~\ref{THM:main-antiRmasey-Turan-Gap}] 
    Fix integers $m \ge r \ge 2$ and an $m$-vertex $r$-graph $F$. 
    % Let $\varepsilon \coloneqq \left(2m|F|\right)^{-1/2}$. 
    Let $n$ be a sufficiently large integer and $t$ be an integer satisfying 
    \begin{align*}
        1 \le t 
        \le \sqrt{\frac{\mathrm{ex}(n,F) - \mathrm{ex}\left(n, \{F\}\cup F\oplus F\right)}{2m |F| \binom{n-1}{r-1}}}.
    \end{align*}
    Let $N \coloneqq \mathrm{ex}(n,tF)+2$. 
    Suppose to the contrary that there exists a surjective edge-coloring $\chi \colon K_{n}^{r} \to [N]$ without any rainbow copy of $(t+1)F$. 
    Let $\mathcal{H} \subset K_{n}^{r}$ be a rainbow (spanning) subgraph with the maximum number of edges. 
    Then, by assumption, $\mathcal{H}$ is $(t+1)F$-free and 
    \begin{align*}
        |\mathcal{H}|
        = N 
        = \mathrm{ex}(n,tF)+2. 
    \end{align*}
    In particular, this implies that 
    \begin{itemize}
        \item there exists a copy of $tF$, denoted by $\Gamma$, in $\mathcal{H}$, and
        \item for every $e\in \Gamma$, there exists a copy of $tF$, denoted by $\Gamma_{e}$, in $\mathcal{H} \setminus \{e\}$. 
    \end{itemize}
    Let $B \coloneqq V(\Gamma) \cup \left(\bigcup_{e\in \Gamma} V(\Gamma_{e})\right)$. 
    Note that 
    \begin{align*}
        |B|
        \le \left(1+t|F|\right)\cdot t m 
        \le 2t^2 m |F|. 
    \end{align*}
    Let $U \coloneqq V(\mathcal{H}) \setminus B$ and $\mathcal{G} \coloneqq \mathcal{H} - B$. 
    The choice of $t$ ensures that 
    \begin{align*}
        |\mathcal{G}|
        \ge \mathcal{H} - |B| \binom{n-1}{r-1}
        & \ge \mathrm{ex}(n,tF)+2 - 2t^2 m |F| \binom{n-1}{r-1} \\
        & \ge \mathrm{ex}(n,tF)+2 - \left(\mathrm{ex}(n,F) - \mathrm{ex}\left(n, \{F\}\cup F\oplus F\right)\right) \\
        & > \mathrm{ex}\left(n, \{F\}\cup F\oplus F\right). 
    \end{align*}
    It is clear that $\mathcal{G}$ is $F$-free, since any copy of $F$ in $\mathcal{G}$ would form a copy of $(t+1)F$ with $\Gamma$. 
    So it follows from the inequality above that $\mathcal{G}$ contains a member in $F \oplus F$ as a subgraph. 
    Let us assume that $F_{e-}^{1}, F_{e'-}^{2} \subset \mathcal{G}$ are copies of two (not necessarily distinct) members in $F_{-}$ such that $F_{e-}^{1} \cup F_{e'-}^{2}$ is a copy of some member in $F \oplus F$. 
    % Let $\phi \colon e \to e'$ be the bijection such that $F_{e-}^{1} \cup F_{e'-}^{2} \cong F\oplus_{\phi}F$. 
    Note that $e$ and $e'$ represent the same $r$-set in $U$. 
    Let us consider the color of $e$ under $\chi$. 
    Since $F_{e-}^{1} \cup F_{e'-}^{2}$ is rainbow, either $F^{1} \coloneqq F_{e-}^{1} \cup \{e\} \cong F$ or $F^{2} \coloneqq F_{e-}^{2} \cup \{e'\} \cong F$ is rainbow. 
    By symmetry, we may assume that $F^{1}$ is rainbow. 
    \begin{itemize}
        \item If $\chi(e)$ is different from all the colors in $\Gamma$, then $\Gamma \sqcup F^{1}$ is a rainbow copy of $(t+1)F$. 
        \item If $\chi(e) = \chi(f)$ for some edge $f\in \Gamma$, then $\Gamma_{f} \sqcup F^{1}$ is a rainbow copy of $(t+1)F$.
    \end{itemize}
    In both cases, we obtain a rainbow copy of $(t+1)F$, which is a contradiction. 
\end{proof}

%%%%%%%%%%%%%%%%%%%%%%%%%%%%%%%%%%%%%%%%%%%%%%
\section{{Proof of Theorem \ref{THM:antiRamsey-reduction}}}\label{SEC:proof-reduction}
In this section, we use Theorem~\ref{THM:stability} to prove Theorem~\ref{THM:antiRamsey-reduction}. 
The lower bound for Theorem~\ref{THM:antiRamsey-reduction} comes from the following fact, which holds for all $r$-graphs. 
\begin{fact}\label{FACT:anti-Ramsey-general-2}
    For every $r$-graph and integers $n,t$ with $1 \le t \le n/v(F)$, 
    \begin{align*}
        \mathrm{ar}(n, (t+2)F)
        \ge \binom{n}{r}-\binom{n-t}{r} + \mathrm{ar}(n-t,2F). 
    \end{align*}
\end{fact}
\begin{proof}[Proof of Fact~\ref{FACT:anti-Ramsey-general-2}]
    Let $N \coloneqq \binom{n}{r}-\binom{n-t}{r} + \mathrm{ar}(n-t,2F) -1$. 
    It suffices to show there exists a surjective edge-coloring $\chi \colon K_{n}^{r} \to [N]$ without any rainbow copy of $(t+2)F$. 
    
    Let $M\coloneqq \mathrm{ar}(n-t,2F) -1$ and fix a surjective edge-coloring $\chi' \colon K_{n-t}^{r} \to [M]$ without any rainbow copy of $2 F$.
    The existence of such a coloring is guaranteed by the definition of $\mathrm{ar}(n-t,2F)$. 
    Let $L \coloneqq [n-t+1, \ldots, n]$ and note that $|L| = t$. Let $\mathcal{G} \coloneqq K_{n}^{r} - L \cong K_{n-t}^{r}$. Note that $|K_{n}^{r}\setminus\mathcal{G}| = \binom{n}{r} - \binom{n-t}{r} = N-M$.  
    Let $\chi \colon K_{n}^{r} \to [N]$ be an edge-coloring such that 
    \begin{itemize}
        \item the induced coloring of $\chi$ on $\mathcal{G}$ is identical to $\chi'$, and 
        \item the induced coloring of $\chi$ on $K_{n}^{r} \setminus \mathcal{G}$ is rainbow and use colors from $[M+1, \ldots, N]$. 
    \end{itemize}
    Note that $\chi$ is an edge-coloring of $K_{n}^{r}$ and uses exactly $N$ colors. 
    Suppose to the contrary that there exists a rainbow copy of $(t+2)F$ under the coloring $\chi$. 
    Then there are at most $t$ copies of $F$ use vertices from $L$, which means that there are at least two copies $F$ that are contained in $\mathcal{G}$. 
    However, this contradicts the fact that the coloring $\chi'$ does not contain any rainbow copy of $2 F$.
\end{proof}
Next, we consider the upper bound in Theorem~\ref{THM:antiRamsey-reduction}. 
A key ingredient in its proof is the following stability theorem concerning near-extremal $r$-graphs without $tF$. 
The proof of this stability theorem is postponed to Section~\ref{SEC:proof-stability}. 
\begin{theorem}\label{THM:stability}
    Let $m \ge  r \ge  2$ be integers. 
    Suppose that $F$ is an $m$-vertex $r$-graph that is smooth and $(c_1,c_2)$-bounded for constants $c_1, c_2$ satisfying $0 < c_1 \le \frac{1-\pi(F)}{12m}$ and $c_2>0$. 
    Then there exists $N_0$ such that the following holds for all $n \ge N_0$ and $t < \frac{\delta \cdot \mathrm{ex}(n,F)}{2m\binom{n-1}{r-1}}$, where 
    $\delta\coloneqq \min\left\{\frac{c_2}{4},\ \frac{1-\pi (F)}{24m}, \frac{1}{100 r m^2}\right\}$. 
    Every $n$-vertex $(t+2)$F-free $r$-graph $\mathcal{H}$ with at least 
    \begin{align*}
         \binom{n}{r} - \binom{n-t}{r} + \mathrm{ex}(n-t, F) -  \frac{1-\pi(F)}{30} \binom{n}{r-1}
    \end{align*}
    edges contains at least $t$ vertices with degree at least $d(n,F) +  \frac{1-\pi(F)}{7m}\binom{n-1}{r-1}$. 
\end{theorem}
The following result, which follows easily from the definition of boundedness (see e.g.~{\cite[Lemma~3.2]{HHLLYZ23a}} and~{\cite[Claim~3.7]{HLLYZ23}} for a proof), will be used in proofs of both Theorems~\ref{THM:antiRamsey-reduction} and~\ref{THM:stability}. 
\begin{lemma}\label{LEMMA:1st-interval-avoid-R}
    Let $n \ge r \ge 2$ be integers, $c_1, c_2,\delta_1, \delta_2 > 0$ be real numbers, and $F$ be a $(c_1, c_2)$-bounded $r$-graph. 
    Suppose that $\mathcal{H}$ is an $n$-vertex $r$-graph and $v\in V(\mathcal{H})$ is a vertex satisfying 
     \begin{align*}
         |\mathcal{H}| 
         \ge (1-c_2+\delta_1) \cdot \mathrm{ex}(n,F)
         \quad\text{and}\quad 
         d_{\mathcal{H}}(v) 
         \ge d(n,F) + (c_1 + \delta_2)\binom{n-1}{r-1}. 
     \end{align*}
     Then for every set $R\subset V(\mathcal{H})\setminus\{v\}$ of size at most $\min\left\{\frac{\delta_1 \cdot \mathrm{ex}(n,F)}{\binom{n-1}{r-1}},\  \frac{\delta_2(n-1)}{r-1} \right\}$, there exists a copy of $F$ in $\mathcal{H}$ that avoids $R$ (i.e. has empty intersection with $R$). 
\end{lemma}
Now we present the proof of Theorem~\ref{THM:antiRamsey-reduction}. 
\begin{proof}[Proof of Theorem~\ref{THM:antiRamsey-reduction}]
    Fix integers $m \ge r \ge 2$ and an $m$-vertex $r$-graph $F$ that is smooth and $(c_1, c_2)$-bounded for constants $c_1, c_2$ satisfying $0 < c_1 < \frac{1-\pi(F)}{12m}$ and $c_2 > 0$. 
    Let $\delta\coloneqq \min\left\{\frac{c_2}{4},\ \frac{1-\pi (F)}{24m}, \frac{1}{100 r m^2}\right\}$ and $n$ be a sufficiently large integer. 
    Simple calculations show that $\min\left\{\frac{\delta \cdot \mathrm{ex}(n,F)}{\binom{n-1}{r-1}}, \frac{\delta (n-1)}{r-1}\right\} = \frac{\delta \cdot \mathrm{ex}(n,F)}{\binom{n-1}{r-1}}$. 
    Let  
    \begin{align*}
        1 
        \le t 
        \le \frac{\delta \cdot \mathrm{ex}(n,F)}{2 m \binom{n-1}{r-1}} - 1
        \quad\text{and}\quad 
        N 
        \coloneqq \binom{n}{r}-\binom{n-t}{r} + \mathrm{ar}(n-t,2F). 
    \end{align*}
    Suppose to the contrary that there exists a surjective edge-coloring $\chi \colon K_{n}^{r} \to [N]$ without any rainbow copy of $(t+2)F$. 
    Let $\mathcal{H} \subset K_{n}^{r}$ be a rainbow (spanning) subgraph with the maximum number of edges. 
    By assumption and Fact~\ref{FACT:antiRamsey-tF-general-bound}, we obtain 
    \begin{align}\label{equ:proof-reduction-1}
        |\mathcal{H}|
        \ge N 
        \ge \binom{n}{r}-\binom{n-t}{r} + \mathrm{ex}(n-t,F)+2. 
    \end{align}
    Let $V\coloneqq V(\mathcal{H}) = V(K_{n}^{r})$, $\alpha \coloneqq \frac{1-\pi(F)}{7m} \ge c_1 + \delta + \frac{1-\pi(F)}{56 m}$, and 
    \begin{align*}
        L'
        \coloneq  \left\{v\in V \colon d_{\mathcal{H}}(v) \ge d(n,F) + \alpha \binom{n-1}{r-1}\right\}. 
    \end{align*}
    Since $\mathcal{H}$ is $(t+2)F$-free, it follows from~\eqref{equ:proof-reduction-1} and Theorem~\ref{THM:stability} that $|L'| \ge t$. 
    Fix a $t$-subset $L \subset L'$ and assume that $L = \{v_1, \ldots, v_t\}$. Let $S \coloneqq V\setminus L$. 
    Since there are exactly $\binom{n}{r} - \binom{n-t}{r}$ edges in $K_{n}^{r}$ that have nonempty intersection with $L$, the number of colors used by edges in $K_{n}^{r}[S] \cong K_{n-t}^{r}$ is at least $N - \left(\binom{n}{r} - \binom{n-t}{r}\right) \ge \mathrm{ar}(n-t,2F)$. 
    % \begin{align*}
    %     N - \left(\binom{n}{r} - \binom{n-t}{r}\right)
    %     \ge \mathrm{ar}(n-t,2F). 
    % \end{align*}
    %
    This implies that $K_{n}^{r}[S]$ contains a rainbow $2F$. 
    Let $B_0 \subset S$ be a vertex set of size $2 m$ such that $K_{n}^{r}[B_0]$ contains a rainbow $2F$. 
    Let $C$ be the collection of colors used by edges in this copy of $2F$ and note that $|C| = 2|F|$ is a constant. 
    Let $\mathcal{H}' \subset \mathcal{H}$ denote the subgraph obtained from $\mathcal{H}$ by removing (at most $2|F|$) edges whose colors are within $C$. 
    To complete the proof, it suffices to find a copy of $tF$ in $\mathcal{H}'[V\setminus B_0]$, as this, together with the $2F$ found above, forms a rainbow copy of $(t+2)F$. 
    \begin{claim}\label{CLAIM:proof-reduction-find-tF}
        There exists a copy of $tF$ in $\mathcal{H}'[V\setminus B_0]$. 
    \end{claim}
    \begin{proof}[Proof of Claim~\ref{CLAIM:proof-reduction-find-tF}]
        We will find $t$ vertex-disjoint copies of $F$, denoted by $F_1, \ldots, F_{t}$, from $\mathcal{H}'[V\setminus B_0]$ inductively. 
        To start with, let $R_1 \coloneqq B_0 \cup \left(L \setminus \{v_1\}\right)$. 
        Notice that 
        \begin{align*}
            |R_1| 
            \le t + 2m 
            \le (2t+2) m 
            \le \frac{\delta \cdot \mathrm{ex}(n,F)}{\binom{n-1}{r-1}}. 
        \end{align*}
        Additionally,~\eqref{equ:proof-reduction-1} implies that $|\mathcal{H}'| \ge \mathrm{ex}(n,F) -2|F| \ge (1-c_2+\delta) \cdot \mathrm{ex}(n,F)$, and the definition of $L'$ implies that $d_{\mathcal{H}'}(v_1) \ge d(n,F) + \alpha \binom{n-1}{r-1} - 2|F| \ge d(n,F) + (c_1 + \delta) \binom{n-1}{r-1}$. 
        Thus, by Lemma~\ref{LEMMA:1st-interval-avoid-R}, there exists a copy of $F$, denoted by $F_1$, in $\mathcal{H}'$ that avoids $R_1$. 
        Suppose that we have found $i$ vertex-disjoint copies of $F$, namely $F_1, \ldots, F_i$, from $\mathcal{H}'$ for some $i\le t-1$. 
        Let $R_{i+1} \coloneqq \bigcup_{1\le j \le i}V(F_j) \cup \left(L \setminus \{v_{i+1}\}\right) \cup B_0$ and note that  
        \begin{align*}
            |R_{i+1}| 
            \le t + 2m + i m 
            \le (2t+2) m 
            \le \frac{\delta \cdot \mathrm{ex}(n,F)}{\binom{n-1}{r-1}}, 
        \end{align*}
        Similar to the argument above, we have $|\mathcal{H}'| \ge (1-c_2+\delta) \cdot \mathrm{ex}(n,F)$ and $d_{\mathcal{H}'}(v_{i+1}) \ge d(n,F) + (c_1 + \delta) \binom{n-1}{r-1}$. 
        So by Lemma~\ref{LEMMA:1st-interval-avoid-R}, there exists a copy of $F$, denoted by $F_{i+1}$, in $\mathcal{H}'$ that avoids $R_{i+1}$. 
        The choice of $t$ ensures that we can repeat this process $t$ times, hence obtaining $t F \subset \mathcal{H}'[V\setminus B_0]$. 
    \end{proof}
    Claim~\ref{CLAIM:proof-reduction-find-tF} completes the proof of Theorem~\ref{THM:antiRamsey-reduction}. 
\end{proof}

%%%%%%%%%%%%%%%%%%%%%%%%%%%%%%%%%%
\section{Proof of Proposition~\ref{PROP:turan-edge-critical}}\label{SEC:proof-Fano-minus}
In this section, we prove Proposition~\ref{PROP:turan-edge-critical}. 
We will need the following stability theorem for the proof of Proposition~\ref{PROP:turan-edge-critical}~\ref{PROP:turan-edge-critical-3}. 

\begin{theorem}[F\"uredi--Simonovits~\cite{FS05}, Keevash--Sudakov~\cite{KS05}]\label{THM:Fano-stability}
    For every $\varepsilon>0$ there is $\delta >0$ and $N_0$ such that every $\mathbb{F}$-free $3$-graph $\mathcal{H}$ on $n \ge N_0$ vertices with at least $(3/4-\delta) \binom{n}{3}$ edges has bipartition $V(\mathcal{H})=V_1 \cup V_2$ such that $|\mathcal{H}[V_1]| + |\mathcal{H}[V_2]| < \varepsilon n^3$. 
\end{theorem}
\begin{proof}[Proof of Proposition~\ref{PROP:turan-edge-critical}]
    Let $F$ be an $r$-graph. For every integer $k \ge 1$ the $k$-blow-up of $F$, denoted by $F[k]$, is the $r$-graph obtained from $F$ by replacing each vertex with a set of size $k$ and each edge with the corresponding complete $r$-partite $r$-graph. 
    It is well-known that $\pi(F) = \pi(F[k])$ for every $k \ge 1$ (see e.g.~\cite{Kee11}). 
    Observe that every $F\oplus F$-free $r$-graph must be $F_{-}[2]$-free (see Figure~\ref{Fig:Fano-plane} for example). Thus, $\pi(F\oplus F) \le \pi(F_{-}[2]) = \pi(F_{-})$, which implies Proposition~\ref{PROP:turan-edge-critical}~\ref{PROP:turan-edge-critical-1}. 

    Proposition~\ref{PROP:turan-edge-critical}~\ref{PROP:turan-edge-critical-2} follows easily from the following two theorems by Lazebnik--Ustimenko--Woldar~\cite{LUW95} and Bondy--Simonovits \cite{BS74}, respectively.
    \begin{itemize}
        \item $\mathrm{ex}(n,C_{2k}) = \Omega(n^{1+2/(3k-2)})$ for every $k \ge 2$, and 
        \item $\mathrm{ex}(n,C_{2k}\oplus C_{2k}) = \mathrm{ex}(n,C_{4k-2}) = O(n^{1+1/(2k-1)})$ for every $k \ge 2$. 
    \end{itemize}

    Next, we prove Proposition~\ref{PROP:turan-edge-critical}~\ref{PROP:turan-edge-critical-3}. 
    It suffices to show that there exist  $\delta>0$ and $N_0$ such that the following holds for all $n\ge N_0 \colon$ 
    Every $n$-vertex $3$-graph with at least $(1/8-\delta) n^3$ edges contains a copy of $\mathbb{F}_{-}$. 

    Suppose to the contrary that this is not true. 
    Let $\delta>0$ to be sufficiently small and $n$ to be sufficiently large. 
    Let $\mathcal{H}$ be an $n$-vertex $\mathbb{F}_{-}$-free $3$-graph with at least $(1/8-\delta) n^3$ edges. 
    By Theorem~\ref{THM:Fano-stability}, there exists a bipartition $V(\mathcal{H}) = V_1 \cup V_2$ such that the bipartite subgraph $\mathcal{G} \coloneqq \mathcal{H} \setminus \left(\mathcal{H}[V_1] \cup \mathcal{H}[V_2]\right)$ satisfies 
    \begin{align}\label{THM:number-edges-H'}
        |\mathcal{G}|
        = |\mathcal{H}| - |\mathcal{H}[V_1]| - |\mathcal{H}[V_2]|
        \ge \left(\frac{1}{8}-\delta\right)n^3 - \varepsilon n^3
        \ge \left(\frac{1}{8}-2\varepsilon\right)n^3, 
    \end{align}
    where $\varepsilon = \varepsilon(\delta) > 0$ is the constant guaranteed by Theorem~\ref{THM:Fano-stability} (and by enlarging $\varepsilon$ if necessary, we may assume that $\varepsilon \ge \delta$). 

    Let $x_i \coloneqq |V_i|/n$ for $i\in \{1,2\}$. 
    Inequality~\eqref{THM:number-edges-H'} and simple calculations show that $\left|x_i - 1/2 \right| \le 2\sqrt{\varepsilon}$ for $i\in \{1,2\}$. 

    \begin{claim}\label{CLAIM:proof-Fano-1}
        There exist three vertices $u_1, u_2, u_3 \in V_1$ such that 
        \begin{align*}
            \left|\binom{V_2}{2} \cap L_{\mathcal{G}}(u_i) \right| 
            \ge (1-110 \varepsilon) \binom{x_2 n}{2}
            \quad\text{for}\quad 
            i\in \{1,2,3\}. 
        \end{align*}
    \end{claim}
    \begin{proof}[Proof of Claim~\ref{CLAIM:proof-Fano-1}]
        Let $\mathcal{B}$ denote the collection of all triples that have nonempty intersection with both $V_1$ and $V_2$. 
        Simple calculations shows that $|\mathcal{B}| \le n^3/8$. 
        Since $\mathcal{G} \subset \mathcal{B}$, it follows from~\eqref{THM:number-edges-H'} that 
        \begin{align*}
            |\mathcal{B} \setminus \mathcal{G}|
            \le \frac{n^3}{8} - \left(\frac{1}{8}-2\varepsilon\right)n^3
            = 2\varepsilon n^3. 
        \end{align*}
        Suppose to the contrary that Claim~\ref{CLAIM:proof-Fano-1} fails. Then we would have 
        \begin{align*}
            |\mathcal{B} \setminus \mathcal{G}|
            \ge \left(x_1 n - 2\right) \left(\binom{x_2 n}{2} - (1-110 \varepsilon) \binom{x_2 n}{2}\right)
            \ge \frac{n}{3} \cdot 110 \varepsilon \cdot \binom{n/3}{2}
            > 2\varepsilon n^3, 
        \end{align*}
        a contradiction. 
    \end{proof}
    Let $u_1, u_2, u_3 \in V_1$ be three vertices guaranteed by Claim~\ref{CLAIM:proof-Fano-1}, and let 
    \begin{align*}
        G 
        \coloneqq \binom{V_2}{2} \cap L_{\mathcal{G}}(u_1) \cap L_{\mathcal{G}}(u_2) \cap L_{\mathcal{G}}(u_3). 
    \end{align*}
    By definition, $|G| \ge \binom{x_2 n}{2} - 3 \cdot 110 \varepsilon \binom{x_2 n}{2} > \frac{(x_2 n)^2}{3}$. 
    So it follows Tur\'{a}n's theorem that there exist four vertices $\{u_4, u_5, u_6, u_7\} \subset V_2$ that induce a copy of $K_4$ in $G$. 
    However, this implies that the induced subgraph of $\mathcal{H}$ on $\{u_1, \ldots, u_7\}$ contains a copy of $\mathbb{F}_{-}$, a contradiction.   
\end{proof}

%%%%%%%%%%%%%%%%%%%%%%%%%%%%%%%%%%%%%%%%%%%%%
\section{Proof of Theorem~\ref{THM:stability}}\label{SEC:proof-stability}
In this section, we present the proof of Theorem~\ref{THM:stability}.
The following preliminary results will be useful.
\begin{fact}\label{FACT:Binomal-Inequ-1}
    Suppose that integer $n,t,r\ge 1$ satisfy $t \le \frac{n-r}{5r+1}$. 
    Then 
    \begin{align*}
        \binom{n-t}{r}
        \ge \frac{1}{e^{1/5}}\binom{n}{r}. 
    \end{align*}
    % In particular, $\binom{n}{r} - \binom{n-t}{r} \ge t \binom{n-t}{r} \ge \frac{t}{e^{1/5}}\binom{n}{r}$. 
\end{fact}
\begin{proof}[Proof of Fact~\ref{FACT:Binomal-Inequ-1}]
    For every $i \in [t]$ it follows from $t \le \frac{n-r}{5r+1}$ that $\frac{n-i}{n-i-r} \le \frac{n-t}{n-t-r} \le 1+\frac{1}{5t}$. 
    Therefore, 
    \begin{align*}
        \binom{n}{r}
        = \prod_{i=0}^{t-1}\frac{n-i}{n-i-r} \binom{n-t}{r}
        \le \left(1+\frac{1}{5t}\right)^{t} \binom{n-t}{r} 
        \le e^{1/5} \binom{n-t}{r},  
    \end{align*}
    proving Fact~\ref{FACT:Binomal-Inequ-1}. 
\end{proof}
%
% \begin{fact}[{\cite[Lemma~3.3]{HLLYZ23}}]\label{FACT:Binomal-Inequ-2}
%     Suppose that $n,m,r$ are integers satisfying $m \le n/r -1$. Then 
%     \begin{align}\label{equ:binomal-ineq}
%     \binom{n}{r}-\binom{n-m}{r}
%     =\sum_{i=1}^{r}\binom{m}{i}\binom{n-m}{r-i}
%     \le 2m\binom{n-m}{r-1}.
%     \end{align}
% \end{fact}
%
% The following lemma says that $d(n,F)$ is smooth for every $F$. 
%
\begin{fact}[{\cite[Lemma~3.5]{HLLYZ23}}]\label{FACT:smooth-degree}
    Let $F$ be an $r$-graph. 
    For every $n$ and $t \le  n/r -1$ we have 
    \begin{align*}
        \left| d(n,F) - d(n-t, F) \right| \le 4t \binom{n-t}{r-2}.
    \end{align*}
\end{fact}
\begin{lemma}\label{LEMMA:smooth-bound}
    Suppose that $F$ is a smooth $r$-graph. Then for sufficiently large $n$ and for every positive integer $t \le n$, 
    \begin{align*}
        \left|\mathrm{ex}(n,F) - \mathrm{ex}(n-t,F) - t \cdot d(n,F)\right|  
        \le \left(\frac{1-\pi(F)}{8m} t + \frac{4(r-1)t^2}{n}\right) \binom{n}{r-1}. 
    \end{align*}
\end{lemma}
\begin{proof}[Proof of Lemma~\ref{LEMMA:smooth-bound}]
    We may assume that $t \le n/2$ since otherwise the inequality is trivial. 
    Let $\Delta \coloneqq \left|\mathrm{ex}(n,F) - \mathrm{ex}(n-t,F) - t \cdot d(n,F)\right|$. 
    It follows from Fact~\ref{FACT:smooth-degree} and the smoothness of $F$ that 
    \begin{align*}
        \Delta
        & \le \sum_{i=1}^{t} \left|\mathrm{ex}(n-i+1,F) - \mathrm{ex}(n-i,F) - d(n-i, F)\right|
        + \sum_{i=1}^{t} \left|d(n-i, F) - d(n,F)\right| \\
        & \le \sum_{i=1}^{t} \left|\delta(n-i+1,F) - d(n-i, F)\right| 
        + \sum_{i=1}^{t} 4i \binom{n-1}{r-2} \\
        & \le \sum_{i=1}^{t} \frac{1-\pi(F)}{8m} \binom{n}{r-1} + 4t^2 \binom{n-1}{r-2} 
        = \left(\frac{1-\pi(F)}{8m} t + \frac{4t^2(r-1)}{n}\right) \binom{n}{r-1},  
    \end{align*}
    proving Lemma~\ref{LEMMA:smooth-bound}. 
\end{proof}
Next, we present the proof of Theorem~\ref{THM:stability}. 
\begin{proof}[Proof of Theorem~\ref{THM:stability}]
    Fix integers $m \ge r \ge 2$ and an $m$-vertex $r$-graph $F$ that is smooth and $(c_1, c_2)$-bounded for constants $c_1, c_2$ satisfying $0 < c_1 \le \frac{1-\pi(F)}{12m}$ and $c_2 > 0$.
    Let  
    \begin{align*}
        \delta \coloneq  \min\left\{\frac{c_2}{4},\ \frac{1-\pi (F)}{24m}, \frac{1}{100 r m^2}\right\}.
    \end{align*}
    Let $n$ be a sufficiently large integer and $t$ be a positive integer satisfying $t < \frac{\delta \cdot \mathrm{ex}(n,F)}{2 m\binom{n-1}{r-1}}$. 
    Simple calculations show that 
    \begin{align*}
        t 
        \le \frac{\delta \cdot \mathrm{ex}(n,F)}{2 m\binom{n-1}{r-1}} - 1
        = \min\left\{\frac{\delta \cdot \mathrm{ex}(n,F)}{2 m\binom{n-1}{r-1}} -1,\  \frac{\delta (n-1)}{2m(r-1)}-1,\ \frac{n}{100 r m^2}\right\}. 
    \end{align*}
    Let $\mathcal{H}$ be an $n$-vertex $(t+2)F$-free $r$-graph with 
    \begin{align}\label{equ:stability-H-edge-lower-bound}
        |\mathcal{H}| 
        \ge  \binom{n}{r} - \binom{n-t}{r} + \mathrm{ex}(n-t, F) - \frac{1-\pi(F)}{30} \binom{n}{r-1}. 
    \end{align}
    %
    %Then $ c_1+\delta < \alpha$. 
    Let $V\coloneqq V(\mathcal{H})$, $\alpha \coloneqq \frac{1-\pi(F)}{7m} >\frac{1-\pi(F)}{12m} + \frac{1-\pi(F)}{24m} \ge c_1 + \delta$, 
    \begin{align*}
        % \alpha \coloneqq \frac{1-\pi(F)}{7m} > c_1 + \delta, \quad 
        L
        \coloneq  \left\{v\in V \colon d_{\mathcal{H}}(v) \ge d(n,F) + \alpha \binom{n-1}{r-1}\right\}, 
        \quad\text{and}\quad 
        S
        \coloneq  V\setminus L.
    \end{align*}
    Let $\ell \coloneqq L$ and suppose to the contrary that $t - \ell \ge 1$. 
    \begin{claim}\label{CLAIM:stability-HS}
        The $r$-graph $\mathcal{H}[S]$ is $(t+2-\ell)F$-free. 
    \end{claim}
    \begin{proof}[Proof of Claim~\ref{CLAIM:stability-HS}]
        The proof is similar to that of Claim~\ref{CLAIM:proof-reduction-find-tF}. 
        Suppose to the contrary that $(t+2-\ell)F \subset \mathcal{H}[S]$. 
        Let $R_0 \subset S$ be a set of size $(t+2-\ell)m$ such that $(t+2-\ell)F \subset \mathcal{H}[R_0]$. 
        We will find $\ell$ vertex-disjoint copies of $F$ from $\mathcal{H}[V\setminus R_0]$ inductively. This will contradict the assumption that $\mathcal{H}$ is $(t+2)F$-free. 
        Assume that $L = \{v_1, \ldots, v_{\ell}\}$. 
        To start with, let $R_1 \coloneqq R_0 \cup \left(L \setminus \{v_1\}\right)$. 
        Notice that 
        \begin{align*}
            |R_1| 
            \le (t+2-\ell)m + t 
            \le (2t+2) m 
            \le \frac{\delta \cdot \mathrm{ex}(n,F)}{\binom{n-1}{r-1}}. 
        \end{align*}
        Additionally,~\eqref{equ:stability-H-edge-lower-bound} implies that $|\mathcal{H}| \ge \mathrm{ex}(n,F) - \frac{1-\pi(F)}{30} \binom{n}{r-1} \ge (1-c_2+\delta) \cdot \mathrm{ex}(n,F)$, and the definition of $L$ implies that $d_{\mathcal{H}}(v_1) \ge d(n,F) + \alpha \binom{n-1}{r-1} \ge d(n,F) + (c_1 + \delta) \binom{n-1}{r-1}$. 
        Thus, by Lemma~\ref{LEMMA:1st-interval-avoid-R}, there exists a copy of $F$, denoted by $F_1$, in $\mathcal{H}$ that avoids $R_1$. 
        Suppose that we have found $i$ vertex-disjoint copies of $F$, namely $F_1, \ldots, F_i$, from $\mathcal{H}$ for some $i\le \ell-1$. 
        Let $R_{i+1} \coloneqq \bigcup_{1\le j \le i}V(F_j) \cup \left(L \setminus \{v_{i+1}\}\right) \cup R_0$. 
        Since 
        \begin{align*}
            |R_{i+1}| 
            \le i m + t + (t+2-\ell)m
            \le (2t+2) m 
            \le \frac{\delta \cdot \mathrm{ex}(n,F)}{\binom{n-1}{r-1}}, 
        \end{align*}
        similar to the argument above,  by Lemma~\ref{LEMMA:1st-interval-avoid-R}, there exists a copy of $F$, denoted by $F_{i+1}$, in $\mathcal{H}$ that avoids $R_{i+1}$. 
        The choice of $t$ ensures that we can repeat this process $\ell$ times, hence obtaining $\ell F \subset \mathcal{H}[V\setminus R_0]$. 
    \end{proof}
    Let $\{F_1, \ldots, F_{\hat{t}}\}$ be a maximum collection of vertex-disjoint copies of $F$ in $\mathcal{H}[S]$ and let $B \coloneqq \bigcup_{i\in [\hat{t}]}V(F_i)$. 
    Notice that $\mathcal{H}[S\setminus B]$ is $F$-free. 
    By Claim~\ref{CLAIM:stability-HS}, we know that $|B| = m \hat{t} \le m (t+1-\ell)$. 
    So it follows from the definition of $L$ that 
    \begin{align*}
        |\mathcal{H}[S]|
        & \le \sum_{v\in B}d_{\mathcal{H}}(v) + |\mathcal{H}[S\setminus B]| \\
        & \le m (t+1-\ell) \left(d(n,F) + \alpha \binom{n-1}{r-1}\right) + \mathrm{ex}\left(n-t-m (t+1-\ell), F\right). 
    \end{align*}
    On the other hand, since there are at most $\binom{n}{r} - \binom{n-\ell}{r}$ edges in $\mathcal{H}$ that have nonempty intersection with $L$, it follows from~\eqref{equ:stability-H-edge-lower-bound} that 
    \begin{align*}
        |\mathcal{H}[S]| 
        & \ge |\mathcal{H}| - \left(\binom{n}{r} - \binom{n-\ell}{r}\right) \\
        & = \binom{n-\ell}{r} - \binom{n-t}{r} + \mathrm{ex}(n-t,F) - \frac{1-\pi(F)}{30} \binom{n-1}{r-1}. 
    \end{align*}
    Recall that $\alpha = \frac{1-\pi(F)}{7m}$. 
    Therefore, to establish a contradiction, it suffices to show that 
    \begin{align*}
        \Delta 
        & \coloneqq \left(\binom{n-\ell}{r} - \binom{n-t}{r} + \mathrm{ex}(n-t,F) - \frac{1-\pi(F)}{30} \binom{n-1}{r-1}\right) \\
        & \quad - \left(m (t+1-\ell) \left(d(n,F) + \frac{1-\pi(F)}{7m} \binom{n-1}{r-1}\right) + \mathrm{ex}\left(n-t-m (t+1-\ell), F\right) \right)
    \end{align*}
    is positive. 
    Notice that $\Delta = \Delta_1 + \Delta_2$, where 
    \begin{align*}
        \Delta_1
        & \coloneqq \mathrm{ex}(n-t,F) - \mathrm{ex}\left(n-t -m (t+1-\ell), F\right) - m (t+1-\ell) \cdot d(n,F), \quad\text{and}\quad\\
        \Delta_2 
        & \coloneqq \binom{n-\ell}{r} - \binom{n-t}{r} - \frac{\left(1-\pi(F)\right)(t+1-\ell)}{7} \binom{n-1}{r-1} - \frac{1-\pi(F)}{30} \binom{n-1}{r-1}.   
        % \Delta_3 
        % & \coloneqq m (t+1-\ell) \left(d(n-t,F) - d(n,F) \right). 
    \end{align*}
    First, let us consider $\Delta_2$. 
    Since $\binom{n-\ell}{r} - \binom{n-t}{r} \ge (r-\ell) \binom{n-t}{r-1}$ and $t \le \frac{n}{100 r m^2}\le \frac{n-(r-1)}{5(r-1)+1}$, it follows from Fact~\ref{FACT:Binomal-Inequ-1} that 
    \begin{align}\label{equ:stability-Delta2}
        \Delta_2
        & \ge (t-\ell)\binom{n-t}{r-1} - \frac{t+1-\ell}{7} \binom{n}{r-1} - \frac{1}{30} \binom{n}{r-1}  \notag \\
        & \ge \frac{t-\ell}{e^{1/5}} \binom{n}{r-1} - \frac{t+1-\ell}{7} \binom{n}{r-1} - \frac{1}{30} \binom{n}{r-1}.  
        % \ge \left(\frac{t-\ell}{e^{1/5}} -  \frac{t+1-\ell}{7}\right) \binom{n}{r-1}. 
    \end{align}
    Next, we consider $\Delta_1$. 
    If $F$ is degenerate, then, by the theorem of Erd\H{o}s~\cite{Erdos64} (see also~{\cite[Theorem~A.1]{HHLLYZ23a}} for a more precise form), $d(n,F) = r\cdot \mathrm{ex}(n,F)/n \le m \cdot n^{r-1-c_{F}}$ for some constant $c_F > 0$. Combining with the assumption that $n$ is sufficiently large, we obtain 
    \begin{align*}
        \Delta_1 
        & = \mathrm{ex}(n-t,F) - \mathrm{ex}\left(n- t- m (t+1-\ell), F\right) - m (t+1-\ell) \cdot d(n,F) \\
        & \ge - m (t+1-\ell) \cdot d(n,F) 
        \ge - (t+1-\ell) m^2 \cdot n^{r-1-c_{F}}
        \ge - \frac{29(t+1-\ell)}{200} \binom{n}{r-1}. 
    \end{align*}
    If $F$ is nondegenerate, then we decompose $\Delta_1$ as $\Delta_1 = \Delta_{1,1} + \Delta_{1,2}$, 
    where 
    \begin{align*}
        \Delta_{1,1}
        & \coloneqq \mathrm{ex}(n-t,F) - \mathrm{ex}\left(n-t -m (t+1-\ell), F\right) - m (t+1-\ell) \cdot d(n-t,F), 
        \quad\text{and}\quad \\
        \Delta_{1,2}
        & \coloneqq  m (t+1-\ell) \cdot \left(d(n-t,F) - d(n,F) \right). 
    \end{align*}
    Applying Lemma~\ref{LEMMA:smooth-bound} to $\Delta_{1,1}$, we obtain  
    \begin{align*}
        \Delta_{1,1} 
        & \ge -  \left(\frac{\left(1-\pi(F)\right)m(t+1-\ell)}{8m} + \frac{4(r-1)m^2(t+1-\ell)^2}{n}\right) \binom{n}{r-1}  \\
        & \ge - \left(\frac{t+1-\ell}{8} + \frac{4(r-1)m^2(t+1-\ell)^2}{n}\right) \binom{n}{r-1}  \\
        & \ge - \left(\frac{t+1-\ell}{8} + \frac{t+1-\ell}{20}\right) \binom{n}{r-1}, 
    \end{align*}
    where the last inequality follows from the assumption that $t \le \frac{n}{100 r m^2}$. 
    Applying Fact~\ref{FACT:smooth-degree} to $\Delta_{1,2}$, we obtain  
    \begin{align*}
        \Delta_{1,2}
        & \ge - 4m(t+1-\ell) t \binom{n-1}{r-2} \\
        & = - (t+1-\ell) \frac{4 m t (r-1)}{n} \binom{n}{r-1} 
         \ge - \frac{t+1-\ell}{20} \binom{n}{r-1}. 
    \end{align*}
    Summing up $\Delta_{1,2}$ and $\Delta_{1,2}$, we obtain 
    \begin{align*}
        \Delta_{1}
        \ge - \frac{9(t+1-\ell)}{40} \binom{n}{r-1}. 
    \end{align*}
    In both cases, we have $\Delta_1 \ge - \frac{9(t+1-\ell)}{40} \binom{n}{r-1}$. Combining with~\eqref{equ:stability-Delta2}, we obtain 
    \begin{align*}
        \Delta 
        & = \Delta_1 + \Delta_2  \\
        & \ge \frac{t-\ell}{e^{1/5}} \binom{n}{r-1} - \frac{t+1-\ell}{7} \binom{n}{r-1} - \frac{1}{30} \binom{n}{r-1} - \frac{9(t+1-\ell)}{40} \binom{n}{r-1}\\
        & \ge \left(\frac{1}{e^{1/5}} -  \frac{2}{7} - \frac{1}{30}  - \frac{9}{20} \right) \binom{n}{r-1}
        > \frac{1}{25} \binom{n}{r-1}.  
    \end{align*}
    % completing the proof of Theorem~\ref{THM:stability}. 
    Here we used the assumption that $t - \ell \ge 1$. 
\end{proof}

%%%%%%%%%%%%%%%%%%%%%%%%%%%%%%%%%%%%%%%%%%%%%%%%%
\section{Concluding remarks}
Recall the definition of edge-sensitive from Section~\ref{SEC:Intro-Turan-Gap}. 
Theorem~\ref{THM:main-antiRmasey-Turan-Gap} motivates the following two questions on Tur\'{a}n problems. 
\begin{problem}
    Let $r \ge 2$ be an integer. 
    Characterize the family of edge-sensitive $r$-partite $r$-graphs. 
    In particular, is $K_{t,t}$ edge-sensitive for $t \ge 3$? 
\end{problem}

\begin{problem}
    Let $r \ge 3$ be an integer. 
    Characterize the family of $r$-graphs $F$ such that $\pi(F) > \pi(F_{-})$. 
    In particular, is it true that $\pi(K_{\ell}^{r}) > \pi(K_{\ell}^{r-})$ for all $\ell > r \ge 3$? 
\end{problem}
%%%%%%%%%%%%%%%%%%%%%%%%%%%%%%%%%%%%%%%%%%%%%%%%%%
\section*{Acknowledgement}
XL would like to thank Jie Ma and Tianchi Yang for discussions related to Definition~\ref{DDF:turan-edge-critical}. 
%%%%%%%%%%%%%%%%%%%%%%%%%%%%%%%%%%%%%%%%%%%%%%%%%%
\bibliographystyle{alpha}%abbrv
\bibliography{antiRamsey}
%%%%%%%%%%%%%%%%%%%%%%%%%%%%%%%%%%%%%%%%%%%%%%%%%
\begin{appendix}
\section*{Definitions for hypergraphs in Tables~\ref{tab:turan-edge-critical.} and~\ref{tab:turan-edge-critical-2.}}
\begin{itemize}
    \item A graph $F$ is \textbf{edge-critical} if there exists an edge $e\in F$ such that $\chi(F-e) < \chi(F)$. 
    \item Fix a graph $F$,  
        the \textbf{expansion} $H_{F}^{r}$ of $F$ is the $r$-graphs obtained from $F$ by adding a set of $r-2$ new vertices into each edge of $F$, and moreover, these new $(r-2)$-sets are pairwise disjoint. 
    \item Given an $r$-graph $F$ with $\ell+1$ vertices, 
        the \textbf{expansion} $H^{F}_{\ell+1}$ of $F$ is the $r$-graph obtained from $F$ by adding, for every pair $\{u,v\}\subset V(F)$ that is not contained in any edge of $F$, an $(r-2)$-set of new vertices, and moreover, these $(r-2)$-sets are pairwise disjoint. 
    \item We say a tree $T$ is an \textbf{Erd\H{o}s--S\'{o}s tree} if it satisfies the famous Erd\H{o}s--S\'{o}s conjecture on trees. 
    The \textbf{$(r-2)$-extension} $\mathrm{Ext}(T)$ of a tree $T$ is  
    \begin{align*}
        \mathrm{Ext}(T) \coloneqq \left\{e\cup A \colon e \in T\right\}, 
    \end{align*}
    where $A$ is a set of $r-2$ new vertices that is disjoint from $V(T)$.   
    An $r$-graph $F$ is an \textbf{extended tree} if $F = \mathrm{Ext}(T)$ for some tree. 
    \item The ($r$-uniform) \textbf{generalized triangle} $\mathbb{T}_r$ is the $r$-graph with vertex set $[2r-1]$ and edge set  
    \begin{align*}
        \left\{\{1,\ldots,r-1, r\}, \{1,\ldots, r-1, r+1\}, \{r,r+1, \ldots, 2r-1\}\right\}. 
    \end{align*}
    \item Let $\mathcal{C}^{2r}_{3}$ (the expanded triangle) denote the $2r$-graph with vertex set $[3r]$ and edge set 
    \begin{align*}
        \left\{\{1,\ldots, r, r+1, \ldots, 2r\}, \{r+1, \ldots, 2r, 2r+1, \ldots, 3r\}, \{1,\ldots, r, 2r+1, \ldots, 3r\}\right\}. 
    \end{align*}
     \item The \textbf{Fano plane} $\mathbb{F}$ is the $3$-graph with vertex set $\{1,2,3,4,5,6,7\}$ and edge set
    \begin{align*}
        \{123,345,561,174,275,376,246\}. 
    \end{align*}
    \item Let $\mathrm{F}_7$ ($4$-book with $3$-pages) denote the $3$-graph with vertex set $\{1,2,3,4,5,6,7\}$ and edge set 
    \begin{align*}
        \left\{1234, 1235, 1236, 4567\right\}. 
    \end{align*}
    \item Let $\mathbb{F}_{4,3}$ denote the $4$-graph with vertex set $\{1,2,3,4,5,6,7\}$ and edge set
    \begin{align*}
        \left\{1234, 1235, 1236, 1237, 4567\right\}. 
    \end{align*}
    \item Let $\mathbb{F}_{3,2}$ denote the $3$-graph with vertex set $\{1,2,3,4,5\}$ and edge set
    \begin{align*}
        \{123,124,125,345\}. 
    \end{align*}
    \item The $3$-graph $K_{4}^{3}\sqcup K_{3}^{3}$ has vertex set $\{1,2,3,4,5,6,7\}$ and edge set 
    \begin{align*}
        \{123,124,234,567\}. 
    \end{align*}
    \item The $r$-graph $M_{k}^{r}$ ($r$-uniform $k$-matching) is the $r$-graph consisting of $k$ pairwise disjoint edges. 
    \item The $r$-graph $L_{k}^{r}$ ($r$-uniform $k$-sunflower) is the $r$-graph consisting of $k$ edges $e_1, \ldots, e_k$ such that for all $1 \leq i < j \leq k$, it holds that $e_i \cap e_j = \{v\}$ for some fixed vertex $v$.
    % \item The $r$-graph $K_r^{r-1}$ is the $r$-graph with vertex set $[r]$ and edge set 
    % \begin{align*}
    %     \{\{1,\dots,r-1\},\{1,\dots,r-2,r\},\dots,\{2,3,\dots,r\}\}.
    % \end{align*}
    \item The $r$-graph $K_\ell^{r-}$ is the $r$-graph obtained from $K_{\ell}^{r}$ by removing one edge.
    \item The tight cycle $C_k^3$ is the 3-graph with vertex set $\{0,1,\dots,k-1\}$ and edge set 
    \begin{align*}
       \{\{i,i+1,i+2\}\Mod{k} \colon 0\le i\le k-1\}.
    \end{align*}
    \item The $r$-graph $C_k^{3-}$ is obtained from $C_k^3$ by removing one edge.
\end{itemize}
\end{appendix}
%%%%%%%%%%%%%%%%%%%%%%%%%%%%%
\end{document}